\documentclass[12pt]{article}
\usepackage{amsmath}
\usepackage{amsfonts}
\usepackage{amsthm}

\newtheorem{thm}{Theorem}[section]
\newtheorem{ex}{Example}
\newtheorem{lemma}[thm]{Lemma}
\newtheorem{prop}[thm]{Proposition}
\newtheorem{corollary}[thm]{Corollary}
\newtheorem{defin}{Definition}[section]

\newtheorem*{question}{Question}

\DeclareMathOperator{\rank}{rk}
\DeclareMathOperator{\Ann}{Ann}
\DeclareMathOperator{\Hom}{Hom}
\title{Subdirect products of finitely presented metabelian groups}
\author{ J. R. J. Groves}
\begin{document}

\maketitle

\section{Introduction}

The initial motivation for the work presented here came after reading  \cite{metab}. That paper, in turn, was motivated by the work of Bridson {\it et al} on the finite presentation of subdirect products of limit groups. The paper \cite{metab} may be seen as showing that the cases for limit groups and for metabelian groups are very different.  Here we attempt to take further the argument for metabelian groups. Whereas it may be said that subdirect products of limit groups are only rarely finitely presented, we have attempted to show that, in some sense, subdirect products of finitely presented metabelian groups are almost always finitely presented. In fact, our approach demands that we consider a minor generalisation of subdirect product, the {\em virtual subdirect product}. We define it in the next section.

This attempt is largely, but not completely, successful. It is possible to make a statement along these lines under reasonably wide-ranging conditions (see Theorem A and the discussion following its proof). But in the exceptional cases where these conditions fail, so does both the letter and the spirit of the theorem. In these cases it is easy to guarantee that there will be many non-finitely presented subdirect products and it appears to be an open question as to whether, in every case,  there is even one finitely presented subdirect product. (See Theorem B and the Question in Section 2.) 

This work began in discussions with Chuck Miller. I thank him for these and for several later helpful discussions.

\section{Statement of results}

In the following, we denote the dual $\Hom(L, \mathbb R)$ of a group L by $L^*$.  Suppose that $G_1,\dots,G_n$ are finitely presented metabelian groups and $G$ is their direct product. Let $S$ be a subgroup of $G$. Then the inclusion of $S$ into $G$ induces a projection $G^*\longrightarrow S^*$; we shall denote the kernel of this by $S^{\circ}$. We shall say that $S$ is a {\em virtual subdirect product} of the $G_i$ if, for each $i$,  the natural projection of $G$ onto $G_i$ maps $S$ onto a subgroup of finite index in $G_i$.

We shall associate the subgroup $S$ of the $G_i$  with the subspace $S^{\circ}$ of $G^*$. We shall fix the codimension $k$ of this subspace; it is also the rank of the image of $S$ in the abelianisation of $G$. The subspaces of a fixed dimension of $G^*$ form the points of a  Grassmanian. (For more on Grassmanians and extra detail for this discussion, see the next section.) We show that the collection $\mathcal G_k^*$ of points in the Grassmanian which correspond to at least one virtual  subdirect product is the complement of a proper subvariety and so  itself has the structure of a quasi-projective variety (See Lemma \ref{vsp1}). Further, if two virtual subdirect products $S$ and $T$ correspond to the same point of $\mathcal G_k^*$ then one is finitely presented if and only if the other one is finitely presented  (See Lemma \ref{samefp}).

We are almost ready to present the main result of the paper. Observe that, if $m$ is the maximal rank of any of the abelianisations of the $G_i$ and if $S$  is a virtual subdirect product, then $k \ge m$. We shall henceforth assume this inequality.
\newtheorem*{thma}{Theorem A}
\begin{thma}
If either
\begin{enumerate}
\item $k\ge m+2$ or
\item the derived group of $G$ has finite exponent
\end{enumerate}
then those points of $\mathcal G_k^*$ which do not correspond to finitely presented groups belong to a proper subvariety of smaller dimension.
\end{thma}
The assumptions given here are not the only assumptions under which the conclusion can be derived. We discuss this briefly after the proof of the Theorem in Section 5.

It remains to consider what else can be said about the set of points of  $\mathcal G_k^*$ which correspond to finitely presented groups, particularly in the cases where the conclusion of Theorem A may not hold.

\newtheorem*{thmb}{Theorem B}
\begin{thmb} \begin{enumerate}
\item The set of points of $\mathcal G^*_k(\mathbb Q)$ which correspond to finitely presented groups is an open subset of $\mathcal G^*_k(\mathbb Q)$. 

\item Suppose that $n=2$ and that the abelianisations of  $G_1$ and $G_2$ have the same torsion-free rank $m$. If either $G_1\cong G_2$ or if $m\le2$ then there exists a point of $\mathcal G^*_m(\mathbb Q)$ which corresponds to a finitely presented group. (Thus the open set of the previous part of the theorem is non-empty).

\item If at least two $G_i$ are non-polycyclic and if $m \le k<\rank(G)$ then the set of points of $\mathcal G^*_k(\mathbb Q)$ which correspond to non-finitely presented virtual subdirect products is non-empty. 

\item The conclusion of Theorem A does not hold if and only if the set of points of $\mathcal G^*_k(\mathbb Q)$ which correspond to non-finitely presented groups  contains a non-empty open subset.
\end{enumerate}
\end{thmb}

We have not been able to establish that the open set of points corresponding to finitely presented virtual subdirect products is non-empty in every case. Here is a more specific question relating to this problem.
\begin{question} Suppose that $G_1$ and $G_2$ are finitely presented metabelian groups and that $A$ is an abelian group isomorphic to both $(G_1)_{\text{ab}}$ and  $(G_2)_{\text{ab}}$. Do there always exist surjective homomorphisms $\phi_i:G_i \rightarrow A$ so that the fibre product of this pair of homomorphisms is finitely presented? \end{question}
Note that Theorem B implies that the rank of $A$ cannot be $2$.
\section{ Grassmanians}
Fix natural numbers $N\ge k$ and fix a vector space $V$ of dimension $N$.  The Grassmanian $\mathcal G(k, N)$ is an algebraic variety having points in one-to-one correspondence with the set of subspaces of $V$ having dimension $k$. Our main reference for Grassmanians is \cite{harris}; in particular, Lecture 6 and Example 11.42. 

There are two possible topologies that can be used with the (real points of) the Grassmanian; the Zariski topology and the topology derived from the natural manifold structure. When we use a topological term, such as `open', without qualification then we mean the second, or `manifold', topology. If we want to refer to the first topology, we append the name `Zariski'.

We need the following facts about Grassmanians. They can be deduced easily from the discussion in \cite{harris}.
\begin{equation}
\parbox{12cm}{\begin{enumerate} \label{Grass}
\item[a)]$\mathcal G(k,N)$ is a smooth irreducible variety of dimension $k(N-k)$;
\item[b)] the set of real points---$\mathcal G(k,N)(\mathbb R)$--- is a manifold;
\item[c)]  let $W$ be a fixed subspace of $V$ having dimension $m$; then the set $\{U\in \mathcal G(k,N): \dim(U\cap W) \ge l\}$ is a subvariety of $\mathcal G(k,N)$; it is proper (and so of smaller dimension) if $m+k-N<l$;
\item[d)] if $W$ is a subspace of dimension $N-k$ of $V$ then the set of complements of $W$ is a dense open subset of $\mathcal G(k,N)$ which we shall denote by $\mathcal N_W$; it can be parametrised by the set of all $k \times N$ row-echelon matrices of the form $\begin{bmatrix} I&A\end{bmatrix}$: here the last $N-k$ columns of the matrix correspond to a basis of $W$; the row-space of the matrix is the required element of $\mathcal G(k,N)$.
\end{enumerate}}
\end{equation}

Our aim is to investigate subspaces of $G^*$ having {\em co-dimension} $k$. For this we must consider rational points of $\mathcal G(N-k,N)$; we shall denote the variety  $\mathcal G(N-k,N)$ by $\mathcal  G_k$. Thus each point of $\mathcal G_k(\mathbb Q)$ corresponds to a rational subspace of $G^*$ and so is of the form $S^{\circ}$ for some subgroup $S$ of $G$. We shall (as above) call $k$ the {\em rank} of $S$; it can be seen alternatively as the torsion-free rank of the image of $S$ in the abelianisation of $G$. 

We show in Proposition \ref{vsp2} that those points of $\mathcal G_k(\mathbb Q)$ which do not correspond to any  virtual subdirect product of the $G_i$ must 
form the rational points of a proper subvariety of $\mathcal G_k$. The complement of this subvariety is a Zariski-open set and so defines a quasi-projective variety, which we shall denote by $\mathcal G^*_k$. Thus the points of $\mathcal G^*_k(\mathbb Q)$ are of the form $S^{\circ}$ for some virtual subdirect product $S$ of the $G_i$.

\section{Bieri-Strebel invariants}

Almost none of this section is new. We have gathered various results relating to the Bieri-Strebel invariant and put them in a form suitable for this exposition.

In this section, let $A$ be a finitely generated abelian group and let $M$ be a finitely generated $\mathbb ZA$-module. The set $A^*= \Hom(A,\mathbb R)$ is a real vector space of dimension equal to the torsion-free rank of $A$.  
The following geometric invariant  was introduced by Bieri and Strebel in \cite{bsval}.

 \begin{defin} The subset $\Sigma_A(M)$ of $A^*$ is defined as the set of $\chi\in A^*$ so that there exists an element $\alpha$ of the annihilator of $M$ in $\mathbb ZA$ of the form $1+\sum_a n_aa$ with $n_a\in \mathbb Z$ and $\chi(a)>0$. \end{defin}
 
 We have made some variations to the usage in \cite{bsval}. The set we denote by $\Sigma_A(M)$ is there denoted by $\Sigma_M(A)$ and the notation $\Sigma$ is  used there for the set of equivalence classes of such $\chi$ where two elements of $A^*$ are regarded as {\it equivalent} if one is a positive scalar multiple of the other. Thus, in \cite{bsval}, $\Sigma$ is naturally regarded as a subset of the $(n-1)$-sphere whereas here we shall regard it as a conical subset of $n$-dimensional Euclidean space. The practical differences are minor.
 
The structure of $\Sigma$ or, more usually, its complement $\Sigma^c$, is best derived from other subsets of $A^*$.

\begin{defin} Let $R$ denote $\mathbb ZA/\Ann_{\mathbb ZA}(M)$ with natural epimorphism $\sigma: \mathbb ZA\rightarrow R$. Let $v$ be a valuation of $\mathbb Z$. Then $\Delta_v(M)$ is defined to be the set of all $\chi\in A^*$ so that there is a valuation $w$ of $R$ so that the composite map $v\circ \sigma$ agrees with $v$ on $\mathbb Z$ and agrees with $\chi$ on $A$. \end{defin} 

The following is a consequence of Theorems A and B of \cite{bgberg}. By a {\it polyhedron} in Euclidean space  we mean a finite union of finite intersections of affine half-planes. A finite intersection of affine half-planes is a convex set  and its {\it dimension} is the dimension of the affine space it spans. In general, the dimension of  a polyhedron is the maximum of the dimensions of the convex polyhedra of which it is a union. Similarly, the transcendence degree of a commutative ring is the maximum of the transcendence degrees of its integral domain quotients.
 
 \begin{thm}[Bieri and Groves \cite{bgberg}] \label{bg1}The set $\Delta_v(M)$  is a polyhedron of dimension equal to the transcendence degree of $R$. \end{thm}
 
Observe that if we add to a polyhedron $\Lambda\subseteq A^*$ all elements equivalent to one of its elements then we finish up with a conical polyhedron $[\Lambda]$ which has dimension at most one more than the original dimension. 
 
 \begin{thm}[Bieri and Groves \cite{bgberg}] \label{bg2}The set $\Sigma^c(M)$ is the union of finitely many sets of the form $[\Delta_v(M)]$ where $v$ is a valuation of $\mathbb Z$. \end{thm}
 
 Thus, if $d$ is the transcendence degree of $R$, then $\Sigma^c(M)$ is a conical polyhedron of dimension either $d$ or $d+1$. The former will happen when each  $\Delta_v(M)$ used in making $\Sigma^c(M)$ is already a cone.
 
 If $B\le A$ then $M$ can also be considered as a $\mathbb ZB$-module; in this case, we denote it by $M_B$. The next result follows by combining Proposition 1.2 of \cite{bsgeomnilp} with the fact, clear from Theorem A of \cite{bgberg}, that $\Sigma^c$ is the closure of its discrete points.
 
 \begin{prop} \label{projsigma}Let $\iota$ denote the embedding of $B$ in $A$ and $\iota^*: A^*\rightarrow B^*$ the restriction map. Then $\iota^*$ induces a surjective map from $\Sigma^c(M)$ onto $\Sigma^c(M_B)$. \end{prop}
 
 The next result appears as  \cite[Corollary 4.5]{bsgeomab}.
 \begin{prop} [Bieri and Strebel] \label{modfg}$M$ is finitely generated as $\mathbb ZB$-module if and only if the kernel $B^{\circ}$ of $ \iota^* : A^* \rightarrow B^*$ has zero intersection with $\Sigma^c(M)$. \end{prop}
 
 We say that $M$ is {\em tame} (sometimes {\em 2-tame}) if $\Sigma^c(M)$ contains no antipodal points. The following is \cite[Theorem A]{bsval}
 
 \begin{thm}[Bieri and Strebel] \label{bsmain}If $G$ is an extension of $M$ by $A$ then $G$ is finitely presented if and only if $M$ is tame. \end{thm}
 
  Note, in particular, that the last result, together with the main result of  \cite{bgfg}, implies that if $G$ is finitely presented then $M\otimes M$ is finitely generated as $A$-module {\em via} the diagonal action. This implies a restriction on the size of the module $M$. The following Proposition and its corollary are very similar to results in \cite{gw}; but the specific result we need does not appear there.
 
 \begin{prop}\label{trdeg}If $M$ is tame as $\mathbb ZA$-module then the transcendence degree of $\mathbb ZA/J$, for any prime ideal $J$ containing the annihilator of $M$,  is at most half of the (torsion-free) rank of $A$. \end{prop}

\begin{proof}

We refer to \cite{bsgeomab} for a discussion of basic properties of the Bieri-Strebel invariant. We have
$$\Sigma^c(M)=\Sigma^c(\mathbb ZA/\Ann(M))= \bigcap \Sigma^c(\mathbb ZA/J)$$
where the intersection is taken over the prime ideals $J$ minimal over $\Ann(M)$. Let $J$ be one of these prime ideals and suppose that $\mathbb ZA/J$ has transcendence degree $d$.  It clearly suffices to show that $d\le \rank(A)/2$. 

Observe that  $\mathbb ZA/J$ is also tame. Thus, by the main result of \cite{bgfg},  $\mathbb ZA/J \otimes_{\mathbb Z} \mathbb ZA/J$ is finitely generated under the diagonal action of $A$. Otherwise put, $\mathbb ZA/J \otimes_{\mathbb Z} \mathbb ZA/J$ is integral over the image of the homomorphism $\delta$ obtained by composing the maps
$$ \mathbb ZA \rightarrow \mathbb ZA \otimes_{\mathbb Z} \mathbb ZA \rightarrow \mathbb ZA/J \otimes_{\mathbb Z} \mathbb ZA/J$$
where the former map is the diagonal map that satisfies $a\mapsto a\otimes a$ for $a\in A$. As $\mathbb ZA/J \otimes_{\mathbb Z} \mathbb ZA/J$ is clearly a cyclic $\mathbb ZA \otimes_{\mathbb Z} \mathbb ZA$-module, it follows that $\mathbb ZA/J \otimes_{\mathbb Z} \mathbb ZA/J$ is integral over $\delta(\mathbb ZA)$. In particular, these two rings have the same transcendence degree. 

The transcendence degree of $\mathbb ZA/J \otimes_{\mathbb Z} \mathbb ZA/J$ is exactly $2d$ and therefore the transcendence degree of  $\delta(\mathbb ZA)$ is exactly $2d$. But the transcendence degree of  $\delta(\mathbb ZA)$ is no more than the transcendence degree of  $\mathbb ZA$ which is easily seen to be $\rank(A)$. The proof is complete.
\end{proof}
\begin{corollary} If $M$ is tame then the dimension of $\Sigma^c(M)$ is at most $\rank(A)/2 +1$. \end{corollary}

\begin{proof} It follows from the proposition that the transcendence degree $d$ of \\$\mathbb ZA/\Ann(M)$ is at most $\rank(A)/2$. By Theorem \ref{bg1}, the dimension of any set of the form $\Delta_v(M)$ is at most $d$ and the dimension of the cone generated by $\Delta_v(M)$ is at most $d+1$. Thus, by Theorem \ref{bg2}, the dimension of $\Sigma^c(M)$  is at most $d+1$ which is, in turn, at most, $\rank(A)/2+1$.\end{proof}

The argument for the following appears first in \cite{bsfg}; we repeat it here because it is very short and elegant.
\begin{lemma}[Bieri and Strebel] \label{notall} If $A$ has torsion-free rank greater than 1 and if $M$ is tame then $\Sigma^c(M) \cup -\Sigma^c(M) \neq A^*$. \end{lemma}
\begin{proof} We must consider the projection of $\Sigma^c(M)$ onto the (unit) sphere; call this $\Omega$. Then $\Omega$ and $-\Omega$ are closed and disjoint subsets of the sphere. As the sphere (in dimension greater than zero)  is connected, their union is therefore not the whole sphere. \end{proof}

\section{Proof of Theorem A}
Let $\pi_i$ denote the projection $G=G_1\times\dots\times  G_n \rightarrow G_i$. Then, because $\pi_i$ has finite cokernel, it induces an injection $\pi_i^*: G_i^*\rightarrow G^*$.

\begin{lemma} \label{sigmagdash}
$$\Sigma^c_G(G')= \cup_i \pi_i^*\left(\Sigma_{G_i}^c(G_i')\right)$$ \end{lemma}
\begin{proof}
By standard properties of the $\Sigma$-invariant, (see, for example, \cite{bsgeomab}), 
$$\Sigma^c_G(G')= \Sigma^c_G(\times_i G'_i) = \cup _i \Sigma^c_G(G'_i).$$
Because each of the $G_j$ with $j\neq i$ act trivially on $G_i'$ we also have\\  $\Sigma^c_G(G'_i)= \pi_i^*\left(\Sigma_{G_i}^c(G_i')\right)$ and the proof is complete.
\end{proof}

Henceforth, we shall identify $G_i^*$ with its image under the projection $\pi_i^*$ and in particular, we shall identify $ \pi_i^*\left(\Sigma_{G_i}^c(G_i')\right)$ with $ \Sigma_{G_i}^c(G_i')$.  Also recall that, if $S$ is a subgroup of $G$, then $S^{\circ}$ denotes the kernel of the restriction map $G^* \longrightarrow S^*$.  Observe that, by a minor extension of Proposition 5 of \cite{metab}, any virtual subdirect product of the $G_i$ is finitely generated.

\begin{lemma} \label{samefp} Suppose $S,T\le G$ are subgoups of $G$.  Then $S^{\circ}=T^{\circ}$ if and only if the isolator of $SG'$ in $G$ coincides with the isolator of $TG'$ in $G$. If $S,T$ are virtual subdirect products  and $S^{\circ}=T^{\circ}$, then  $S$ is finitely presented if and only if $T$ is finitely presented. \end{lemma}

\begin{proof} Observe firstly that $S^{\circ}=(SG')^{\circ}$ as every $\rho\in G^*$ must be zero on $G'$. Also, if $\rho\in G^*$ is zero on a power of an element of $G$ then it must be zero on the element itself. Thus, denoting the isolator of $SG'$ by $i(SG')$, we must have $(SG')^{\circ}=(i(SG'))^{\circ}$. Thus it suffices to show that, for $S$ and $T$ isolated subgroups containing $G'$, we have $S^{\circ}=T^{\circ}$ if and only if $S=T$.

This is easily verified. One way is to embed $G/G'$ into $G/G' \otimes \mathbb Q$. Then observe that there is a bijective correspondence between isolated subgroups of $G/G'$ and subspaces of $G/G' \otimes \mathbb Q$. Further, the correspondence $U\rightarrow U^{\circ}$ is bijective for subspaces $U$ of a finite dimensional vector space.

Suppose then that $S^{\circ}=T^{\circ}$. It follows that $SG'$ is finitely presented if and only if $TG'$ is finitely presented since they have a common finite extension. Thus it remains to show that (say) $S$ is finitely presented if and only if $SG'$ is finitely presented. To show this, observe that $S$ is finitely presented if and only if $S'$ is tame as $S/S'$-module and $SG'$ is finitely presented if and only if $G'$ is tame as $SG'/G'$-module. The latter can be re-stated as `$G'$ is tame as $S/(S\cap G')$-module' or, because $(S\cap  G')/S'$ acts trivially on $G'$, that ``$G'$ is tame as $S/S'$-module''. 
Thus we need to show that $S'$ is tame as $S/S'$-module if and only if $G'$ is tame as $S/S'$-module.

Thus it will suffice to prove that:

$$ \Sigma_{S/S'}^c (S')=\Sigma_{S/S'}^c (G').$$

Observe firstly, that because $S$ is now assumed to be a virtual subdirect product, for each $i$ we have that  $\pi_i(S)=H_i$ has finite index in $G_i$. It follows easily that $G_i'/H_i'$ is finitely generated as abelian group. Thus
$$ \Sigma_{S/S'}^c (G_i')=\Sigma_{S/S'}^c(G_i'/H_i') \cup \Sigma_{S/S'}^c(H_i') =  \Sigma_{S/S'}^c(H_i').$$
The last equality holds because $G_i'/H_i'$ is finitely generated as abelian group and so the complement of its invariant is zero.

But $H_i'=\pi_i(S')$ is an $S$-module quotient of $S'$ and so  $\Sigma_{S/S'}^c (G_i')=  \Sigma_{S/S'}^c(H_i')$ is a subset of $\Sigma_{S/S'}^c(S')$. Thus $\Sigma_{S/S'}^c(S')\supseteq \cup_i \Sigma_{S/S'}^c(G_i')$. Thus,
$$\Sigma_{S/S'}^c (S')\subseteq \Sigma_{S/S'}^c(G')=\cup_i \Sigma_{S/S'}^c(G_i') \subseteq \Sigma_{S/S'}^c (S'),$$
by Lemma \ref{sigmagdash}.
Thus we have equality throughout and the proof is complete.

\end{proof}
In contrast to the case for finite presentation, it is possible that among the groups corresponding to a single point of the Grassmannian, some groups may have the property of being a virtual subdirect product and some may not. But we can still make a partial recognition of the property in geometric terms.
\begin{lemma} \label{vsp1}Let  $S$  be a subgroup of $G$.\begin{enumerate}
\item  If $S$ is a virtual subdirect product of the $G_i$ then
 $S^{\circ} \cap G_i^*=\{0\}$ for each $i$.
 \item If $S^{\circ} \cap G_i^*=\{0\}$ for each $i$ and $S\ge G'$ then $S$ is a virtual subdirect product of the $G_i$.
 \end{enumerate}\end{lemma}
\begin{proof} 1. Let $\sigma$ denote the embedding of $S$ into $G$. Then $\pi_i\circ\sigma$ has finite cokernel because $\pi_i(S)$ has finite index in $G_i$. Thus the dual map $\sigma^*\circ \pi_i^*$ is injective and this is equivalent to saying that $G_j^* \cap S^{\circ} =\{0\}$.

2. Suppose that $S\ge G'$ and $S$ is not a virtual subdirect product. Then $\left|G_j:\pi_j(S)\right|$ is infinite for some $j$. But $G_j'\le \pi_j(G') \le \pi_j(S)$  and so $G_j/\pi_j(S)$ is infinite and abelian. Thus we can find a non-zero character of $G_j$ which is zero on $\pi_j(S)$. Pulling this character back to a character of $G$ provides a non-zero element of  $G_j^* \cap S^{\circ} $.\end{proof}

\begin{prop} \label{vsp2}If $k\ge m$ then the set of  points of$\mathcal G_k(\mathbb Q)$ which do not correspond to any  virtual subdirect product of the $G_i$, is a proper subvariety of $\mathcal G_k(\mathbb Q)$ having smaller dimension. \end{prop}
\begin{proof} If $k\ge m$ then $k \ge \dim(G_i^*)$ for each $i$. Thus the set of subspaces in $\mathcal G_k(\mathbb Q)$ which meet a single subspace $G_i^*$ is a proper subvariety by (c) of (\ref{Grass}). Thus the set of subspaces which meet at least one of the $G_i^*$ is also a proper subvariety.  By Lemma \ref{vsp1} and Lemma \ref{samefp}, this is just the set of points of  $\mathcal G_k(\mathbb Q)$ which do not correspond to any virtual subdirect product of the $G_i$.\end{proof}

\begin{lemma} \label{fpvsp}Let $\Gamma$ denote the set of points of $G^*$ of the form $\chi+\psi$ where $\chi, \psi \in \Sigma^c_G(G')$. Then the virtual subdirect product $S$ is finitely presented if and only if $\Gamma\cap S^{\circ}=\{0\}$.\end{lemma}
\begin{proof}  By Lemma \ref{samefp}, we can assume that $S$ contains $G'$. By Theorem \ref{bsmain}, $S$ will be finitely presented if and only if $G'$ is tame as $S/G'$-module. Let $\sigma^*$ denote the natural surjection $G^*\rightarrow S^*$ and let  $G'_S$ denote $G'$ considered as $S/G'$-module. Then, by Proposition \ref{projsigma},  $\Sigma^c(G'_S)=\sigma^*(\Sigma^c(G'))$. Thus $\Sigma_c(G'_S)$ is  {\em not} tame precisely when there exist $\chi, \psi \in \Sigma_c(G')$ so that $\sigma^*(\chi), \sigma^*(\psi) \neq 0$ and $\sigma^*(\chi)+\sigma^*(\psi) = 0$. Noting that $S^{\circ}$ is the kernel of $\sigma^*$, this happens exactly when $\chi,\psi \notin S^{\circ}$ and $\chi+\psi \in S^{\circ}$. But, by Lemma \ref{sigmagdash}, $\chi\in \Sigma^c(G')$ implies $\chi \in G_i^*$ for some $i$. But then, by Lemma \ref{vsp1},  if also $\chi\in S^{\circ}$ then $\chi=0$. Thus $\chi,\psi \notin S^{\circ}$ is equivalent to $\chi, \psi \neq 0$. Thus $\Sigma^c(G'_S)$  is tame if and only if $\Gamma\cap S^{\circ}=\{0\}$.
\end{proof}

\begin{proof}[Proof of Theorem A]
There is little further to do. By Theorems \ref{bg1} and \ref{bg2}, $\Sigma_{G_i}^c(G_i')$ is a rational polyhedron. By Proposition \ref{trdeg} the dimension of this polyhedron is at most $\rank(G_i)/2 +1$. Thus $\Gamma=\Sigma^c_G(G')+\Sigma^c_G(G')$ is a rational polyhedron of dimension $d=\max_{i\neq j}(\rank(G_i)/2+\rank(G_j)/2) +2\le m+2$ (recall that $m$ is the maximum rank of any $G_i$). If $G$ has finite exponent then each $\Delta_v$ is a cone and so $[\Delta]=\Delta$ has dimension at most $m/2$. Thus, in this case, $d\le m$. Hence $\Gamma$ lies in a finite union $\Lambda$ of rational subspaces of dimension at most $d$.

Suppose that $k\ge d$.  The set of points of $\mathcal G^*_k$ which correspond to non-finitely presented groups is contained in the set of points which have non-trivial intersection with at least one of the subspaces in $\Lambda$. By c) of (\ref{Grass}), this is a finite union of proper subvarieties and so a proper subvariety.
\end{proof}

\noindent {\bf Remark.}  Theorem A is true under much wider conditions than assumed in the statement. The key condition we need is that $d=\dim(\Gamma) \le k$. So, for example, if $\Delta_v (G_i')$ is a cone for each $v$ and each $i$ then the conclusion of the theorem holds even though $G'$ may not be of finite exponent.

\section{Proof of Theorem B} 

\begin{proof}[Proof of part (1)] Firstly observe that, by  Lemma \ref{fpvsp}, a subspace $W=S^{\circ}$ of $G^*$ corresponds to a finitely presented virtual subdirect product $S$ exactly when $W\cap \Gamma=\{0\}$. Since $\Gamma$ is a finite union of convex polyhedral cones, the set of subspaces avoiding $\Gamma$ is therefore a finite intersection of sets which each are the set of subspaces avoiding a single convex polyhedral cone. Thus we must show that these latter sets are open.

Assume that an inner product has been chosen on $G^*$.  We can then represent a convex polyhedral cone  $\Delta$ in the form 
$$\Delta=\{x\in G^*: (x,\eta_i) \ge 0\ \text{ for each } i\}$$
for a set $\{\eta_1, \dots, \eta_r\}$ with $\|\eta_i\|=1$.  We aim to show that the set of subspaces $W$ of codimension $k$, which satisfy $W\cap \Delta=\{0\}$, is open. There is nothing to prove if this set is empty so let  $W$ be one such subspace avoiding $\Delta$ and let $U$ be a complement for $W$ in $G^*$. It will suffice to show that the subspaces in $\mathcal N_U$ which avoid $\Delta$ form an open subset of  $\mathcal N_U$. 

We follow the description of  $\mathcal N_U$ in (d) of (1). In particular each subspace has a basis represented by a reduced row-echelon matrix of the form $\begin{bmatrix} I & A \end{bmatrix}$.

We need to quantify the fact that $W$ meets $\Delta$ only in $0$. Denote $N-k$ by $l$ and let $\{w_1,\dots, w_l\}$ by be a basis of $W$. Let $s=(s_1, \dots, s_l)$ represent a point of the unit sphere in $\mathbb R^l$. Then $w=\sum_{i=1}^ls_i w_i$ represents a point of $W$ which, by assumption, does not lie in $\Delta$. The function $f: s\mapsto \min_j(w, \eta_j)$ is clearly a continuous function on the unit sphere taking values in the real numbers. As no non-zero point of $W$ lies in $\Delta$ , there must, for each $w\in W$ be a $j$ so that $(w, \eta_j)<0$. Hence the function $f$ takes only  negative values. As the sphere is compact, $f$ takes a maximum there and this maximum must be negative. In summary, there exists $\epsilon>0$ so that, for each $s$ on the unit sphere there exists $j$ so that, if $w=\sum_{i=1}^ls_i w_i$ then $(w,\eta_j)<-\epsilon$.

We choose a neighbourhood $\mathcal N$ of $W$ in $\mathcal N_U$ so that, if $W'\in \mathcal N$ and $W'$ is represented by the basis $\{w_1', \dots, w_l'\}$, then $\|w'_i-w_i\|<\epsilon/2l$ for each $i$ with $1\le i\le l$. We must show that, if $0\neq w'=\sum_{i=1}^l s_iw'_i\in W'$, then $w' \notin \Delta$. As $W'$ and $\Delta$ are both cones, there is no harm in assuming that $\|(s_1,\dots,s_l)\|=1$ so that $(s_1,\dots,s_l)$ lies on the unit sphere. Choose $j$ so that $(w,\eta_j)<-\epsilon$. Then we have
\begin{eqnarray*} (w', \eta_j) &=&(w, \eta_j)+(w'-w, \eta_j)\le -\epsilon+ \|w'-w\|\cdot \|\eta_j\|=-\epsilon+ \|w'-w\| \\
&\le& -\epsilon+ \sum_{i=1}^l |s_i|\|w'_i-w_i\|
\le -\epsilon+ \sum_{i=1}^l \|w'_i-w_i\| \\
&\le&-\epsilon+l(\epsilon/2l)=-\epsilon/2<0.\end{eqnarray*}

Thus $w' \notin \Delta$ and so $W' \cap \Delta =\{0\}$. Thus every element of $\mathcal N$ avoids $\Delta$ and so the set of elements of $\mathcal N_U$ which avoid $\Delta$ is open, as required.\end{proof}

\begin{proof}[Proof of part (2)] Consider the dense open set $\mathcal N_{G_1^*}$ of complements to $G_1^*$ in $G^*$. Form a basis for $G^*$ by combining a basis for $G_2^*$ with a basis for $G_1^*$ (in that order). The elements of   $\mathcal N_{G_1^*}$ will then be the row-spaces of a matrix of the form $\begin{bmatrix} I & A \end{bmatrix}$. In particular, the case $A=0$ corresponds to $G_2^*$ itself.

A typical element of the subspace $W_A$ corresponding to the matrix $\begin{bmatrix} I & A\end{bmatrix}$ is of the form $w+ \rho_A(w)$ where $w\in G_2^*$ and $\rho_A$ is the linear transformation $G_2^* \rightarrow G_1^*$ corresponding to the matrix $A$. From this, it is easily verified that $W_A$ is also a complement to $G_2^*$ exactly when $A$ is invertible. Thus those subspaces that correspond to virtual subdirect products are all of the form 
$$W_{\rho}=\{w+\rho(w): w\in G_2^* \}$$
where $\rho:G_2^*\longrightarrow G_1^*$ is invertible.

Thus $W_A\cap \Gamma$ is non-zero exactly when $w+ \rho_A(w) \in \Gamma$ for some $0\neq w\in G_2^*$.  Using the definition of $\Gamma$, this means that $w+ \rho_A(w)=\phi+\psi$ for $\phi,\psi \in \Sigma^c_G(G')$. Using Lemma  \ref{sigmagdash} and the fact that each $\Sigma^c_{G_i}(G_i')$ is tame, we see that we must have $w=\phi\in  \Sigma^c_{G_2}(G_2')$ and $\rho_A(w)=\psi\in  \Sigma^c_{G_1}(G_1')$ . Thus $W_A\cap \Gamma$ is non-zero exactly when 
$$\Sigma^c_{G_1}(G_1') \cap \rho_A(\Sigma^c_{G_2}(G_2'))\neq \{0\}.$$

If $G_1=G_2$ then we can take $A$ to be the negative of the identity matrix and then, by Theorem \ref{bsmain},  the condition above is exactly the condition that $G_i'$ be tame as $G_i$-module. As this is guaranteed by the fact that each $G_i$ is finitely presented, the proof is complete in this case.

The case $m=1$  is straightforward and we omit the proof. Suppose now that $m=2$. We must show how, given two tame polyhedral cones $\Delta_i, i=1,2$ in $\mathbb Q^2$, to construct a linear transformation $\rho$ so that $\Delta_1 \cap \rho(\Delta_2)=\{0\}$.

By Lemma \ref{notall}, we can find $v_i$ so that $\left<v_i\right> \cap \Delta_i=\{0\}$. As the set of such $v_i$ is open, we can assume that $v_1 \neq v_2$. Define an inner product on $V$ by taking $\{v_1, v_2\}$ as an orthonormal basis. Because $\Delta_i$ is closed and $v_i\notin \Delta_i$, there is a minimum value taken for the angle between $v_i$ and $\Delta_i$. Thus there exists $\epsilon_i$ so that $0<\epsilon_i<1$ and 
\begin{equation*}\label{largeangle} \left(\frac{(x,v_i)}{\|x\|}\right)^2\le\epsilon_i\ \text{for all }x\in \Delta_i, x\neq 0.\end{equation*}
Choose $\lambda>0$ so that 
\begin{equation}\label{lambdaval}\lambda^4> \frac{\epsilon_1\epsilon_2}{(1-\epsilon_1)(1-\epsilon_2)} .\end{equation} 

Define a linear transformation $\rho$ by $\rho(v_1)=(1/\lambda)v_1; \rho(v_2)=\lambda v_2$.  Our aim is to argue that any point in $\Delta_1$ is at least a minimum angular distance from $v_1$ and hence is mapped, by $\rho$, to an element sufficiently close to $v_2$ so that we know it cannot belong to $\Delta_2$.

Suppose that $0\neq x\in \Delta_1$; say $x=\alpha v_1+\beta v_2$.  Then 
\begin{equation*}\label{xnotin} \left(\frac{(x,v_1)}{\|x\|}\right)^2 \le\epsilon_1;\quad \text{ that is, } \frac{\alpha^2}{\alpha^2+\beta^2}\le \epsilon_1\end{equation*}
which is equivalent to 
\begin{equation}\label{eq1}\alpha^2 \le \left(\frac{\epsilon_1}{1-\epsilon_1}\right)\beta^2\end{equation}
Observe that $\rho(x)=( \alpha/\lambda)v_1+\beta\lambda v_2$. We will know that $\rho(x)\notin \Delta_2$ if 
\begin{equation*}\label{rhoxnotin} \left(\frac{(\rho(x),v_2)}{\|x\|}\right)^2 >\epsilon_2;\quad \text{ that is, }
 \frac{\beta^2\lambda^4}{\frac{alpha^2}{\lambda^2}+\beta^2\lambda^2}> \epsilon_2\end{equation*}
which is equivalent to
\begin{equation}\label{rhoxnotin}\alpha^2<\beta^2\lambda^4\left(\frac{1-\epsilon_2}{\epsilon_2}\right)\end{equation}

If $x\in\Delta_1$ then we have, from (\ref{lambdaval}) and (\ref{eq1}), that 
$$\alpha^2 \le \left(\frac{\epsilon_1}{1-\epsilon_1}\right)\beta^2<\beta^2\lambda^4\left(\frac{1-\epsilon_2}{\epsilon_2}\right)$$
which shows, via (\ref{rhoxnotin}), that $\rho(x)\notin \Delta_2$.

\end{proof}

\begin{proof}[Proof of part (3)] The set $\Sigma^c_{G_i}(G_i')$ is zero precisely when $G_i$ is polycylic. Thus, if two factors $G_i$ are non-polycyclic, then $\Gamma$ will contain a non-zero ray which is not in any $G_i^*$. As long as $k<\rank(G^*)$ we can find  a subspace of codimension $k$ which avoids the factors $G_i^*$ and which contains the given element of $\Gamma$. By Lemma \ref{fpvsp}, the corresponding virtual subdirect product is not finitely presented.
\end{proof}
\begin{proof}[Proof of part (4)] If the conclusion of Theorem A holds then the set of points corresponding to non-finitely presented groups lies in a proper subvariety and so cannot contain a non-empty open subset.

Suppose now that the conclusion of Theorem A fails. Recall the definition of the set $\Lambda$ in Lemma \ref{fpvsp} and let $d$ denote the dimension of $\Lambda$. The comments following the proof of Theorem A show that the conclusion holds if $d\le k$. Thus we may suppose that $d>k\ge m$. Then $\Gamma$ contains a convex polyhedral cone of dimension $d$ and so there is a linearly independent subset $\mathcal B$ of size $d$ so that every positive linear combination of elements of $\mathcal B$ lies in $\Gamma$.  Let $\mathcal B_1$ be a subset of size $k$ in $\mathcal B$ and let $\mathcal B_2$ be its complement in $\mathcal B$. Let $W$ be the subspace spanned by $\mathcal B_1$. Form a basis $\mathcal C$ for $G^*$ which has $\mathcal B_2$ as its first $d-k$ elements and $\mathcal B_1$ as its last $k$ elements. 

Then $\mathcal N_W$ (see (d) of (1) ) can be parametrised by $(N-k)\times N$ matrices of the form $$\begin{bmatrix} I&A\end{bmatrix}$$
The rows are indexed by the elements of $\mathcal C \setminus \mathcal B_1$ and the columns of $A$ are indexed by the elements of $\mathcal B_1$.

The subset of $N_W$ consisting of all those subspaces in which all entries of the matrix $A$ are (strictly) positive is clearly an open subset of $N_U$ and hence also of $\mathcal G(N-k,N)$. Thus it will suffice to show that a point $U_A$ of $N_W$ which corresponds to a virtual subdirect product must correspond to a non-finitely presented one. For that, it suffices to show, by Lemma \ref{fpvsp}, that $U_A$ meets the set $\Lambda$ in a non-zero element of $G^*$. But $U_A$ is the row-space of $[\begin{bmatrix} I&A\end{bmatrix}$ and, in particular, contains the first row. But this first row is the sum of an element of $\mathcal B_2$ and a positive linear combination of the elements of $\mathcal B_1$. In particular, this first row is a positive linear combination of the elements of $\mathcal B$ and so is in $\Gamma\subseteq \Lambda$. Thus if an element of $N_W$ has all elements of $A$ positive and corresponds to a virtual subdirect product $S$ then $S$ is non-finitely presented. The proof is complete.

\end{proof}


\end{document}